\documentclass[11pt]{article}
\usepackage{enumerate}
\usepackage{amssymb,a4wide,latexsym,makeidx,epsfig,fleqn}
\usepackage{amsthm}
\usepackage{amsmath}
\usepackage{enumerate}
\usepackage{graphicx}
\usepackage{float}
\usepackage{tikz}
\usepackage[colorlinks=true, linkcolor=blue, citecolor=red, urlcolor=blue]{hyperref}
\allowdisplaybreaks[4]
\newtheorem{theorem}{Theorem}[section]

\newtheorem{lemma}[theorem]{Lemma}

\newtheorem{problem}{Problem}

\makeatletter
\AtBeginDocument{
	\renewenvironment{proof}[1][\proofname]
	{
		\par
		\pushQED{\qed}
		\normalfont
		\topsep6pt\relax
		\trivlist
		\item[\hskip\labelsep\normalfont\bfseries
		#1\@addpunct{.}]
		\ignorespaces
	}
	{
		\popQED
		\endtrivlist
		\@endpefalse
	}
}
\makeatother
\begin{document}
\textwidth 150mm \textheight 225mm
\title{Some results on the distance spectral radius and edge-disjoint spanning trees of graphs\thanks{Supported 
by the National Natural Science Foundation of China (No. 12271439).}}
\author{{Yongbin Gao$^{a,b}$, Ligong Wang$^{a,b,}$\footnote{Corresponding author.}}\\
{\small $^a$ School of Mathematics and Statistics, Northwestern
Polytechnical University,}\\ {\small  Xi'an, Shaanxi 710129,
P.R. China.}\\
{\small $^b$ Xi'an-Budapest Joint Research Center for Combinatorics, Northwestern
Polytechnical University,}\\
{\small Xi'an, Shaanxi 710129,
P.R. China. }\\
{\small E-mail: gybmath@163.com, lgwangmath@163.com} }
\date{}
\maketitle
\begin{center}
\begin{minipage}{120mm}
\vskip 0.3cm
\begin{center}
{\small {\bf Abstract}}
\end{center}
{\small Let $\tau(G)$ denote the maximum number of edge-disjoint spanning trees in a connected graph $G$ of order $n$, and let $\rho_D(G)$ denote its distance spectral radius. For an integer $k\ge2$, Fan, He and Zhao [Discrete Appl. Math. 376 (2025) 31--40] obtained a sharp distance spectral radius condition for $\tau(G)\ge k$ when $n\ge2k+6$. In this paper, we fill the gap $2k\le n\le2k+5$ and thus complete the result for all $n\ge2k$. The extremal graph given by Fan, He and Zhao remains valid for $n\ge2k+2$, while we determine the unique extremal graph for each of the orders $n=2k$ and $n=2k+1$. We further obtain sharp distance spectral radius conditions and characterize all extremal graphs under the minimum degree condition $\delta(G)\ge k$ for all $n\ge2k$. Finally, for graphs with the stronger minimum degree condition $\delta(G)\ge6k-4$ and order $n\ge2\delta(G)+2$, we obtain a sharp distance spectral radius condition ensuring $\tau(G)\ge k$ and determine the unique extremal graph.

}

\end{minipage}
\end{center}

\section{Introduction}

In this paper, we consider finite, undirected and simple graphs. Let $G$ be a graph with vertex set $V(G) = \{v_1, v_2,\ldots, v_n\}$ and edge set $E(G)$. We denote the order and size of $G$ by $n=|V(G)|$ and $e(G)=|E(G)|$, respectively. For a vertex $v_i\in V(G)$, let $N_G(v_i)$ and $d_G(v_i)=|N_G(v_i)|$ denote its neighborhood and degree, and let $\delta(G)$ (or simply $\delta$) be the minimum degree of $G$. A graph $G$ is called $d$-regular if every vertex of $G$ has degree $d$. For $X\subseteq V(G)$, let $G[X]$ denote the subgraph of $G$ induced by $X$. For any two disjoint subsets $X, Y \subseteq V(G)$, let $E_G(X, Y)$ be the set of edges connecting a vertex in $X$ to a vertex in $Y$, and denote $e_G(X, Y) = |E_G(X, Y)|$. We sometimes omit the subscript $G$ when there is no ambiguity. For any two vertices $v_i$ and $v_j$ of a connected graph $G$, the distance $d_G(v_i, v_j)$ between  $v_i$ and $v_j$ is the length of a shortest path connecting them. The distance matrix of $G$, denoted by $D(G)$, is the $n \times n$ matrix whose $(v_i, v_j)$-entry is $d_G(v_i, v_j)$(or $d_{ij}$). The largest eigenvalue of $D(G)$, denoted by $\rho_D(G)$, is called the distance spectral radius of $G$. 

The spanning tree packing number of a graph $G$, denoted by $\tau(G)$, is defined as the maximum number of edge-disjoint spanning trees contained in $G$. Seymour proposed the following problem in a private communication with Cioab\v{a} (see \cite{Cioaba2012}).

\begin{problem}\label{prob1}
	Let $G$ be a connected graph. Determine the relationship between $\tau(G)$ and the eigenvalues of $G$.
\end{problem}

Let $A(G)$ be the adjacency matrix of a graph $G$ of order $n$, and let $\lambda_1(G)\ge\lambda_2(G)\ge\cdots\ge\lambda_n(G)$ be its eigenvalues. The largest eigenvalue $\lambda_1(G)$ is called the spectral radius of $G$, denoted by $\rho(G)$. Motivated by Problem \ref{prob1}, Cioab\u{a} and Wong \cite{Cioaba2012} initiated the study of the relationship between $\tau(G)$ and adjacency eigenvalues. They obtained sufficient conditions involving $\lambda_2(G)$ for $d$-regular graphs to satisfy $\tau(G)\ge k$ for $k=2,3$, and proposed a conjecture for $4\le k\le\lfloor d/2\rfloor$. Gu et al. \cite{Gu2016} extended this conjecture to graphs with minimum degree $\delta$, and Liu et al. \cite{Liu2014a} subsequently proved it. More precisely, for an integer $k\ge2$, they showed that if $G$ has minimum degree $\delta\ge2k$ and $\lambda_2(G)<\delta-\frac{2k-1}{\delta+1}$, then $\tau(G)\ge k$.
For the spectral radius, Fan et al. \cite{Fan2023} proved that if $k\ge2$ and $G$ is a connected graph of order $n\ge2\delta+3$ with minimum degree $\delta\ge2k$, then $\rho(G)\ge\rho(B_{n,\delta+1}^{k-1})$ implies $\tau(G)\ge k$, unless $G\cong B_{n,\delta+1}^{k-1}$, where $B_{n,\delta+1}^{k-1}$ is obtained from $K_{\delta+1}\cup K_{n-\delta-1}$ by adding $k-1$ edges joining one vertex in $K_{\delta+1}$ to $k-1$ distinct vertices in $K_{n-\delta-1}$.
More results on the relationship between eigenvalues of graphs and the spanning tree packing number can be found in \cite{Cai2026,Chang2026,Cioaba2022,Duan2020,Fan2025,Gao2026a,Gao2026b,Hong2016,Hu2023,Li2013,Liu2014b,Liu2019,Zhang2026}.

Recently, Fan et al. \cite{Fan2025} obtained sharp distance spectral radius conditions guaranteeing $\tau(G)\ge k$ for general connected graphs and connected balanced bipartite graphs, respectively. We denote by $G_1 \vee G_2$ the join of two disjoint graphs $G_1$ and $G_2$, which is obtained from the union $G_1 \cup G_2$ by adding all edges between $V(G_1)$ and $V(G_2)$. For a graph $G$ and a positive integer $c$, let $cG$ denote the disjoint union of $c$ copies of $G$. For a graph $H$ and a subset of edges $E' \subseteq E(H)$, let $H\setminus E'$ denote the graph obtained from $H$ by deleting the edges in $E'$. Let $K_n$, $P_n$, and $C_n$ denote the complete graph, path, and cycle of order $n$, respectively. The complement of a graph $G$ of order $n$, denoted by $\overline{G}$, is defined by $\overline{G}=K_n\setminus G$. For general connected graphs, Fan et al. \cite{Fan2025} obtained the following result.

\begin{theorem}[\cite{Fan2025}]
	\label{thm:fan}
	Let $k\ge2$ be an integer and let $G$ be a connected graph of order $n\ge2k+6$. If $\rho_D(G)\le\rho_D(K_{k-1}\vee(K_{n-k}\cup K_1))$, then $\tau(G)\ge k$, unless $G\cong K_{k-1}\vee(K_{n-k}\cup K_1)$.
\end{theorem}

A simple graph of order $n\ge2$ containing $k$ edge-disjoint spanning trees must satisfy $k(n-1)\le n(n-1)/2$, and hence $n\ge2k$ is necessary. We show that the sharp distance spectral radius condition in Theorem \ref{thm:fan} remains valid for all $n\ge2k+2$. We also study the two remaining cases $n=2k$ and $n=2k+1$, for which we establish sharp distance spectral radius conditions and characterize the unique extremal graphs. Thus, we obtain a complete characterization of the sharp distance spectral radius conditions for all possible orders $n\ge2k$, as summarized in the following theorem.

\begin{theorem}
	\label{thm:all1}
	Let $k\ge2$ be an integer and let $G$ be a connected graph of order $n\ge2k$. Then the following statements hold.
	
	\begin{enumerate}[(i)]
		\item If $n=2k$ and $\rho_D(G)\le\rho_D(K_n-e)$, where $e\in E(K_n)$, then $\tau(G)\ge k$, unless $G\cong K_n-e$.
		
		\item If $n=2k+1$ and $\rho_D(G)\le\rho_D\bigl(\overline{P_3\cup(k-1)P_2}\bigr)$, then $\tau(G)\ge k$, unless $G\cong\overline{P_3\cup(k-1)P_2}$.
		
		\item If $n\ge2k+2$ and $\rho_D(G)\le\rho_D\bigl(K_{k-1}\vee(K_{n-k}\cup K_1)\bigr)$, then $\tau(G)\ge k$, unless $G\cong K_{k-1}\vee(K_{n-k}\cup K_1)$.
	\end{enumerate}
\end{theorem}

We next consider the minimum degree condition $\delta(G)\ge k$, which is necessary for $\tau(G)\ge k$. The extremal graph $K_{k-1}\vee(K_{n-k}\cup K_1)$ in Theorem \ref{thm:fan} has minimum degree $k-1$, so we seek a stronger distance spectral radius condition under this restriction. The cases $n=2k$ and $n=2k+1$ are already covered by Theorem \ref{thm:all1} $(i)$ and $(ii)$, since their extremal graphs have minimum degree $2k-2\ge k$. We therefore consider $n\ge2k+2$.

\begin{figure}[htbp]
	\centering
	
	\tikzset{
		vertex/.style={
			circle,
			fill=black,
			inner sep=0pt,
			minimum size=5.8pt
		},
		clique/.style={
			draw=black!80,
			fill=black!2,
			line width=0.85pt
		},
		subset/.style={
			draw=blue!65!black,
			fill=white,
			dash pattern=on 3pt off 2pt,
			line width=0.7pt
		},
		graph edge/.style={
			draw=black,
			line width=0.85pt,
			line cap=round
		}
	}
	
	\begin{minipage}[t]{0.48\textwidth}
		\centering
		\vspace{0pt}
		
		\resizebox{0.98\linewidth}{!}{
			\begin{tikzpicture}[font=\small]
				\path[use as bounding box]
				(-3.12,0.55) rectangle (3.12,5.35);
				
				\coordinate (w) at (0,4.80);
				\coordinate (v1) at (-1.20,2.75);
				\coordinate (v2) at (-0.40,2.75);
				\coordinate (vk) at (1.20,2.75);
				
				\draw[clique] (0,2.30)
				ellipse (2.90cm and 1.52cm);
				
				\draw[subset] (0,2.60)
				ellipse (1.90cm and 0.65cm);
				
				\foreach \vertexname in {v1,v2,vk}
				\draw[graph edge] (w) -- (\vertexname);
				
				\foreach \vertexname in {w,v1,v2,vk}
				\node[vertex] at (\vertexname) {};
				
				\node[above=5pt] at (w) {$\omega$};
				
				\node at (-1.20,2.38) {$v_1$};
				\node at (-0.40,2.38) {$v_2$};
				\node at (1.20,2.38) {$v_{k-1}$};
				
				\node[inner sep=0pt] at (0.40,2.75) {$\cdots$};
				
				\node[font=\normalsize] at (0,1.22) {$K_{n-1}$};
			\end{tikzpicture}
		}
		
		\par\vspace{0.2em}
		\caption{The graph $K_{k-1}\vee(K_{n-k}\cup K_1)$}
		\label{fig:join-graph}
	\end{minipage}
	\hfill
	\begin{minipage}[t]{0.48\textwidth}
		\centering
		\vspace{0pt}
		
		\resizebox{0.98\linewidth}{!}{
			\begin{tikzpicture}[font=\small]
				\path[use as bounding box]
				(-3.12,0.55) rectangle (3.12,5.35);
				
				\coordinate (u) at (-1.44,4.80);
				\coordinate (v) at (1.44,4.80);
				
				\coordinate (s1) at (-2.22,2.75);
				\coordinate (s2) at (-1.70,2.75);
				\coordinate (sk) at (-0.66,2.75);
				
				\coordinate (t1) at (0.66,2.75);
				\coordinate (t2) at (1.18,2.75);
				\coordinate (tk) at (2.22,2.75);
				
				\draw[clique] (0,2.30)
				ellipse (2.90cm and 1.52cm);
				
				\draw[subset] (-1.44,2.60)
				ellipse (1.20cm and 0.65cm);
				\draw[subset] (1.44,2.60)
				ellipse (1.20cm and 0.65cm);
				
				\draw[graph edge] (u) -- (v);
				
				\foreach \vertexname in {s1,s2,sk}
				\draw[graph edge] (u) -- (\vertexname);
				
				\foreach \vertexname in {t1,t2,tk}
				\draw[graph edge] (v) -- (\vertexname);
				
				\foreach \vertexname in {u,v,s1,s2,sk,t1,t2,tk}
				\node[vertex] at (\vertexname) {};
				
				\node[above=5pt] at (u) {$u$};
				\node[above=5pt] at (v) {$v$};
				
				\node at (-2.22,2.38) {$s_1$};
				\node at (-1.70,2.38) {$s_2$};
				\node at (-0.66,2.38) {$s_{k-1}$};
				
				\node at (0.66,2.38) {$t_1$};
				\node at (1.18,2.38) {$t_2$};
				\node at (2.22,2.38) {$t_{k-1}$};
				
				\node[inner sep=0pt] at (-1.18,2.75) {$\cdots$};
				\node[inner sep=0pt] at (1.70,2.75) {$\cdots$};
				
				\node at (-1.44,1.72) {$S$};
				\node at (1.44,1.72) {$T$};
				
				\node[font=\normalsize] at (0,1.22) {$K_{n-2}$};
			\end{tikzpicture}
		}
		
		\par\vspace{0.2em}
		\caption{The graph $F_{n,k}$}
		\label{fig:Fnk}
	\end{minipage}
	
\end{figure}

For $n=2k+2\ge6$, let $\mathcal{F}_n$ be the family of graphs obtained from $K_n$ by deleting the edges of vertex-disjoint cycles that together cover all its vertices. Equivalently, $\mathcal{F}_n=\{\overline{C_{n_1}\cup\cdots\cup C_{n_s}}: s\ge1,\ n_i\ge3,\ \sum_{i=1}^{s}n_i=n\}$, where $s,n_1,\ldots,n_s$ are integers.

For $n\ge2k+3$, choose two disjoint subsets $S,T\subseteq V(K_{n-2})$ with $|S|=|T|=k-1$. Let $F_{n,k}$ be the graph obtained from $K_{n-2}$ by adding two new vertices $u$ and $v$, the edge $uv$, all edges joining $u$ to the vertices of $S$ and all edges joining $v$ to the vertices of $T$.

\begin{theorem}
	\label{thm:all2}
	Let $k\ge2$ be an integer and let $G$ be a connected graph of order $n\ge2k+2$ with minimum degree $\delta(G)\ge k$. Then the following statements hold.
	
	\begin{enumerate}[(i)]
		\item If $n\ge2k+3$ and $\rho_D(G)\le\rho_D(F_{n,k})$, then $\tau(G)\ge k$, unless $G\cong F_{n,k}$.
		
		\item If $n=2k+2$ and $\rho_D(G)\le n+1$, then $\tau(G)\ge k$, unless $G\cong H$ for some $H\in\mathcal{F}_n$.
	\end{enumerate}
\end{theorem}

We next consider a stronger minimum degree condition. A matching is a set of edges that have no common endpoints. For integers $k\ge2$, $\delta\ge 6k-4$ and $n\ge2\delta+2$, let $M_{n,\delta+1}^{k-1}$ be the graph obtained from $K_{\delta+1}\cup K_{n-\delta-1}$ by adding a matching of size $k-1$ between the two complete graphs. We obtain the following result.

\begin{theorem}
	\label{thm:strong}
	Let $k\ge2$ be an integer and let $G$ be a connected graph of order $n\ge2\delta+2$ with minimum degree $\delta\ge6k-4$. If $\rho_D(G)\le\rho_D(M_{n,\delta+1}^{k-1})$, then $\tau(G)\ge k$, unless $G\cong M_{n,\delta+1}^{k-1}$.
\end{theorem}

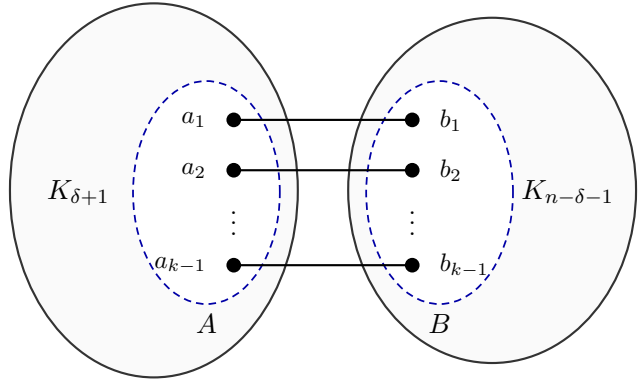
\begin{figure}[htbp]
	\centering
	
	\tikzset{
		vertex/.style={
			circle,
			fill=black,
			inner sep=0pt,
			minimum size=5.8pt
		},
		clique/.style={
			draw=black!80,
			fill=black!2,
			line width=0.85pt
		},
		subset/.style={
			draw=blue!65!black,
			fill=white,
			dash pattern=on 3pt off 2pt,
			line width=0.7pt
		},
		graph edge/.style={
			draw=black,
			line width=0.85pt,
			line cap=round
		}
	}
	
	\resizebox{0.75\linewidth}{!}{
		\begin{tikzpicture}[font=\small]
			\path[use as bounding box] (-6.00,0.25) rectangle (6.00,5.30);
			
			\coordinate (CL) at (-2.35,2.85);
			\coordinate (CR) at ( 2.35,2.85);
			
			\draw[clique] (CL) ellipse (2cm and 2.6cm);
			\draw[clique] (CR) ellipse (2cm and 2.4cm);
			
			\draw[subset] (-1.62,2.82) ellipse (1.02cm and 1.55cm);
			\draw[subset] ( 1.62,2.82) ellipse (1.02cm and 1.55cm);
			
			\coordinate (a1) at (-1.24,3.83);
			\coordinate (a2) at (-1.24,3.13);
			\coordinate (ak) at (-1.24,1.81);
			
			\coordinate (b1) at (1.24,3.83);
			\coordinate (b2) at (1.24,3.13);
			\coordinate (bk) at (1.24,1.81);

			\draw[graph edge] (a1) -- (b1);
			\draw[graph edge] (a2) -- (b2);
			\draw[graph edge] (ak) -- (bk);

			\foreach \vertexname in {a1,a2,ak,b1,b2,bk}
			\node[vertex] at (\vertexname) {};

			\node[left=7pt] at (a1) {$a_1$};
			\node[left=7pt] at (a2) {$a_2$};
			\node[left=7pt] at (ak) {$a_{k-1}$};
			
			\node[right=7pt] at (b1) {$b_1$};
			\node[right=7pt] at (b2) {$b_2$};
			\node[right=7pt] at (bk) {$b_{k-1}$};

			\node[inner sep=0pt] at (-1.24,2.5) {$\vdots$};
			\node[inner sep=0pt] at ( 1.24,2.5) {$\vdots$};

			\node[font=\normalsize] at (-1.62,1.0) {$A$};
			\node[font=\normalsize] at ( 1.62,1.0) {$B$};

			\node[font=\normalsize] at (-3.4,2.82) {$K_{\delta+1}$};
			\node[font=\normalsize] at (3.4,2.82) {$K_{n-\delta-1}$};
			
		\end{tikzpicture}
	}
	
	\par\vspace{0.2em}
	\caption{The graph $M_{n,\delta+1}^{k-1}$}
	\label{fig:Mndelta}
\end{figure}

The rest of the paper is organized as follows. In Section 2, we present some preliminary results. In Section 3, we prove Theorem \ref{thm:all1}. In Section 4, we prove Theorem \ref{thm:all2}. In Section 5, we prove Theorem \ref{thm:strong}. In Section 6, we give some concluding remarks.

\section{Preliminaries}

In this section, we present some preliminary results used in our proofs. For a partition $\mathcal{P}=\{V_1,\ldots,V_p\}$ of $V(G)$ into nonempty parts, let $e_G(\mathcal{P})=\sum_{1\le i<j\le p}e_G(V_i,V_j)$ denote the number of edges with ends in distinct parts. We first recall the classical Tree Packing Theorem of Nash-Williams and Tutte.

\begin{theorem}[\cite{Nash1961,Tutte1961}]
	\label{thm:packing}
	Let $G$ be a connected graph of order at least two and let $k\ge1$ be an integer. Then $\tau(G)\ge k$ if and only if $e_G(\mathcal{P})\ge k(p-1)$ for every partition $\mathcal{P}$ of $V(G)$ into $p\ge2$ nonempty parts.
\end{theorem}

Let $M$ be a real matrix of order $n$, and let $\Pi=\{X_1,\ldots,X_t\}$ be a partition of $\{1,\ldots,n\}$ into nonempty parts. Let $M_{ij}$ for the block with rows indexed by $X_i$ and columns indexed by $X_j$. The quotient matrix of $M$ with respect to $\Pi$ is the matrix $B_\Pi=(b_{ij})_{1\le i,j\le t}$, where $b_{ij}$ is the average row sum of $M_{ij}$. If every block $M_{ij}$ has constant row sum, then $\Pi$ is called an equitable partition and $B_\Pi$ an equitable quotient matrix. For a matrix $M$ with real eigenvalues, let $\lambda_1(M)$ denote its largest eigenvalue.

\begin{lemma}[\cite{Brouwer2012}, \cite{Godsil2001}]
	\label{lem:quotient}
	Let $M$ be a real symmetric matrix. If $B_\Pi$ is an equitable quotient matrix of $M$, then every eigenvalue of $B_\Pi$ is an eigenvalue of $M$. Furthermore, if $M$ is nonnegative and irreducible, then $\lambda_1(M)=\lambda_1(B_\Pi)$.
\end{lemma}

\begin{lemma}[{\cite{Gu2016}}]
	\label{lem:cut}
	Let $G$ be a connected graph with minimum degree $\delta$, and let $U$ be a nonempty proper subset of $V(G)$. If $e_G(U,V(G)\setminus U)\le\delta-1$, then $|U|\ge\delta+1$.
\end{lemma}

We also need the following combinatorial inequality.

\begin{lemma}[\cite{Wei2022}]
	\label{lem:sum}
	Let $a_1,\ldots,a_\ell$ be integers with $\ell\ge2$ and $0\le a_1\le\cdots\le a_\ell$. Let $x_1,\ldots,x_\ell$ be integers satisfying $x_i\ge a_i$ for $1\le i\le\ell$ and $\sum_{i=1}^{\ell}x_i=n$. Then
	\[
	\sum_{i=1}^{\ell}\binom{x_i}{2}
	\le
	\binom{n-\sum_{i=1}^{\ell-1}a_i}{2}
	+\sum_{i=1}^{\ell-1}\binom{a_i}{2}.
	\]
	Moreover, the equality holds if and only if
	$x_i=a_i$ for any $i\in\{1,2,\ldots,l-1\}$ and
	$x_l=n-\sum_{i=1}^{l}a_i$.
\end{lemma}

We next recall some properties of the distance spectral radius. 

\begin{lemma}
	\label{lem:edge}
	Let $e$ be an edge of a connected graph $G$ such that $G-e$ is also connected. Then $\rho_D(G)<\rho_D(G-e)$.
\end{lemma}

The Wiener index of a connected graph $G$ of order $n$ is defined as $W(G)=\sum_{1\le i<j\le n}d_G(v_i,v_j)$. By the Rayleigh quotient principle, $\rho_D(G)\ge\frac{\mathbf{1}^{\mathsf T}D(G)\mathbf{1}}{\mathbf{1}^{\mathsf T}\mathbf{1}}=\frac{2W(G)}{n}$, where $\mathbf{1}$ is the all-ones column vector of order $n$.

We also use the following lower bound due to Zhou and Trinajsti\'{c} \cite{Zhou2007}.

\begin{lemma}[\cite{Zhou2007}]
	\label{lem:square}
	Let $G$ be a connected graph of order $n\ge2$. Then
	\[
	\rho_D(G)^2\ge
	\frac{1}{n}\sum_{i=1}^{n}
	\left(\sum_{j=1}^{n}d_G(v_i,v_j)\right)^2.
	\]
\end{lemma}

\section{Proof of Theorem \ref{thm:all1}}

\begin{lemma}
	\label{lem:complete}
	Let $k\ge1$ be an integer. $\tau(K_{2k})=k$.
\end{lemma}

\begin{proof}
	Since every spanning tree of $K_{2k}$ has $2k-1$ edges, we have $\tau(K_{2k})\le\binom{2k}{2}/(2k-1)=k$.
	
	Let $\mathcal{P}=\{V_1,\ldots,V_p\}$ be any partition of $V(K_{2k})$ into $p\ge2$ nonempty parts. Since $|V_i|\le2k-p+1$ for each $i$, we obtain
	\[
	e_{K_{2k}}(\mathcal{P})
	=\frac12\sum_{i=1}^{p}|V_i|(2k-|V_i|)
	\ge\frac{p-1}{2}\sum_{i=1}^{p}|V_i|
	=k(p-1).
	\]
	By Theorem \ref{thm:packing}, $\tau(K_{2k})\ge k$. Hence $\tau(K_{2k})=k$.
\end{proof}

We now prove Theorem \ref{thm:all1} $(i)$.

\begin{proof}[Proof of Theorem \ref{thm:all1} (i)]
	Suppose that $\tau(G)\le k-1$. Since $n=2k$, Lemma \ref{lem:complete} gives $\tau(K_n)=k$, hence $G$ is not complete. Choose an edge $e\in E(\overline{G})$. Then $G$ is a spanning subgraph of $K_n-e$. By Lemma \ref{lem:edge}, we have $\rho_D(G)\ge\rho_D(K_n-e)$, with strict inequality unless $G\cong K_n-e$. Together with the assumption $\rho_D(G)\le\rho_D(K_n-e)$, we have $G\cong K_n-e$. This completes the proof.
\end{proof}

\begin{lemma}
	\label{lem:odd-bound}
	Let $k\ge2$ be an integer and $n=2k+1$. Then
	\begin{equation}
		\label{eq:odd-bound}
		n+\frac1n
		<\rho_D\left(\overline{P_3\cup(k-1)P_2}\right)
		<n+\frac1n+\frac1{n^2}.
	\end{equation}
\end{lemma}

\begin{proof}
	Let $F_k=\overline{P_3\cup(k-1)P_2}$, and write $P_3=v_1v_2v_3$ in $\overline{F_k}$. For the partition $\{\{v_2\},\{v_1,v_3\},V(F_k)\setminus\{v_1,v_2,v_3\}\}$, the distance matrix $D(F_k)$ has the equitable quotient matrix
	\[
	Q_1=
	\begin{pmatrix}
		0 & 4 & n-3\\
		2 & 1 & n-3\\
		1 & 2 & n-3
	\end{pmatrix}.
	\]
	By Lemma \ref{lem:quotient}, $\rho_D(F_k)=\lambda_1(Q_1)$. By a direct calculation, the characteristic polynomial of $Q_1$ is
	\[
	\varphi_{Q_1}(x)=x^3-(n-2)x^2-(2n+2)x+n-3.
	\]
	For $x\ge n$, we have
	\[
	\varphi_{Q_1}'(x)
	=(x-n)(3x+n+4)+n^2+2n-2>0.
	\]
	Thus $\varphi_{Q_1}(x)$ is strictly increasing on $[n,+\infty)$. Since $n\ge5$, we obtain
	\begin{align*}
		\varphi_{Q_1}\left(n+\frac1n\right)
		&=-1+\frac2{n^2}+\frac1{n^3}<0,\\
		\varphi_{Q_1}\left(n+\frac1n+\frac1{n^2}\right)
		&=\frac2n+\frac4{n^2}+\frac7{n^3}
		+\frac5{n^4}+\frac3{n^5}+\frac1{n^6}>0.
	\end{align*}
	Hence,
	\[
	n+\frac1n
	<\lambda_1(Q_1)=\rho_D(F_k)
	<n+\frac1n+\frac1{n^2}.
	\]
\end{proof}

\begin{lemma}
	\label{lem:integers}
	Let $d_1,\ldots,d_n$ be nonnegative integers satisfying $\sum_{i=1}^{n}d_i=n+1$. If $(d_1,\ldots,d_n)$ is not a permutation of $(2,1,\ldots,1)$, then $\sum_{i=1}^{n}d_i^2\ge n+5$.
\end{lemma}

\begin{proof}
	If $d_i\ge1$ for every $i$, then $\sum_{i=1}^{n}d_i=n+1$ implies that exactly one entry is $2$ and all others are $1$, contrary to the assumption. Thus at least one entry is $0$. Without loss of generality, assume that $d_n=0$.
	
	For every nonnegative integer $d_i$, we have $(d_i-1)(d_i-2)\ge0$, so $d_i^2\ge3d_i-2$. Hence,
	\[
	\sum_{i=1}^{n}d_i^2
	=\sum_{i=1}^{n-1}d_i^2
	\ge3\sum_{i=1}^{n-1}d_i-2(n-1)
	=3(n+1)-2(n-1)
	=n+5.
	\]
	This completes the proof.
\end{proof}

We now prove Theorem \ref{thm:all1} $(ii)$.

\begin{proof}[Proof of Theorem \ref{thm:all1} (ii)]
	Suppose that $\tau(G)\le k-1$. By Theorem \ref{thm:packing}, there exists a partition $\mathcal P=\{V_1,\ldots,V_p\}$ of $V(G)$ into $2\le p\le n=2k+1$ nonempty parts such that
	\begin{equation}
		\label{eq:odd-partition}
		e_G(\mathcal P)\le k(p-1)-1.
	\end{equation}
	Let $n_i=|V_i|$ for $1\le i\le p$. By Lemma \ref{lem:sum} and Inequality \eqref{eq:odd-partition}, we have
	\begin{align*}
		e(G)
		&\le\sum_{i=1}^{p}\binom{n_i}{2}+e_G(\mathcal P)\\
		&\le\binom{n-p+1}{2}+k(p-1)-1\\
		&=2k^2-1-\frac{(p-2)(2k-p+1)}2
		\le2k^2-1,
	\end{align*}
	where the last inequality follows from $2\le p\le2k+1$. Hence,
	\[
	e(\overline{G})
	=\binom{n}{2}-e(G)
	\ge\binom{2k+1}{2}-(2k^2-1)
	=k+1.
	\]
	
	We claim that $e(\overline{G})=k+1$. Suppose that $e(\overline{G})\ge k+2$. Since adjacent vertices have distance one and nonadjacent vertices have distance at least two, we have $W(G)\ge\binom{n}{2}+e(\overline{G})$. By Lemma \ref{lem:odd-bound}, we obtain
	\begin{align*}
		\rho_D(G)
		&\ge\frac{2W(G)}n
		\ge\frac2n\left(\binom{n}{2}+e(\overline{G})\right)\\
		&\ge n-1+\frac{2(k+2)}n
		=n+\frac3n\\
		&>n+\frac1n+\frac1{n^2}
		>\rho_D\left(\overline{P_3\cup(k-1)P_2}\right),
	\end{align*}
	where $n=2k+1\ge5$. This contradicts the assumption.
	Thus $e(\overline{G})=k+1$. Then
	\[
	\sum_{i=1}^{n}d_{\overline{G}}(v_i)
	=2e(\overline{G})
	=2(k+1)
	=n+1.
	\]
	Suppose that the degree sequence of $\overline{G}$ is not a permutation of $(2,1,\ldots,1)$. By Lemma \ref{lem:integers}, we have $\sum_{i=1}^{n}d_{\overline{G}}(v_i)^2\ge n+5$. By Lemma \ref{lem:square}, we obtain
	\begin{align*}
		\rho_D(G)^2
		&\ge\frac1n\sum_{i=1}^{n}
		\left(\sum_{j=1}^{n}d_G(v_i,v_j)\right)^2\\
		&\ge\frac1n\sum_{i=1}^{n}
		\left(n-1+d_{\overline{G}}(v_i)\right)^2\\
		&=(n-1)^2+\frac{2(n-1)(n+1)}n
		+\frac1n\sum_{i=1}^{n}d_{\overline{G}}(v_i)^2\\
		&\ge(n-1)^2+\frac{2(n-1)(n+1)}{n}+\frac{n+5}{n} 
		=n^2+2+\frac3n.
	\end{align*}
	Since $n\ge5$, we have
	\[
	n^2+2+\frac3n-\left(n+\frac1n+\frac1{n^2}\right)^2
	=\frac{n^3-n^2-2n-1}{n^4}>0.
	\]
	Together with Lemma \ref{lem:odd-bound}, we have $\rho_D(G)>\rho_D\left(\overline{P_3\cup(k-1)P_2}\right)$, which is contrary to the assumption. Hence the degree sequence of $\overline{G}$ is a permutation of $(2,1,\ldots,1)$. Let $u$ be the unique vertex of degree two in $\overline{G}$, and let $u_1,u_2$ be its neighbors. Since every other vertex has degree one, $u_1uu_2$ forms a component isomorphic to $P_3$, and the remaining $2k-2$ vertices form $k-1$ components isomorphic to $P_2$. Thus $\overline{G}\cong P_3\cup(k-1)P_2$, and hence $G\cong\overline{P_3\cup(k-1)P_2}$. This completes the proof.
\end{proof}

The following upper bound will be used in the proof of Theorem \ref{thm:all1} $(iii)$.

\begin{lemma}
	\label{lem:fan-bound}
	Let $k\ge2$ and $n\ge2k+2$ be integers. Then
	\begin{equation}
		\label{eq:fan-bound}
		\rho_D\left(K_{k-1}\vee(K_{n-k}\cup K_1)\right)
		<n+3-\frac{4(k+1)}{n}.
	\end{equation}
\end{lemma}

\begin{proof}
	For the partition $\{V(K_{k-1}),V(K_{n-k}),V(K_1)\}$, the distance matrix $D(K_{k-1}\vee(K_{n-k}\cup K_1))$ has the equitable quotient matrix
	\[
	Q_2=
	\begin{pmatrix}
		k-2 & n-k & 1\\
		k-1 & n-k-1 & 2\\
		k-1 & 2(n-k) & 0
	\end{pmatrix}.
	\]
	By Lemma \ref{lem:quotient}, $\rho_D(K_{k-1}\vee(K_{n-k}\cup K_1))=\lambda_1(Q_2)$. By a direct calculation, the characteristic polynomial of $Q_2$ is
	\[
	\varphi_{Q_2}(x)
	=x^3-(n-3)x^2-(5n-3k-3)x-k^2+kn+4k-5n+1.
	\]
	For $x\ge n$, we have
	\[
	\varphi_{Q_2}'(x)
	=(x-n)(3x+n+6)+n^2+n+3k+3>0.
	\]
	Thus $\varphi_{Q_2}(x)$ is strictly increasing on $[n,+\infty)$.
	
	By a direct calculation, we obtain
	\begin{align*}
		\varphi_{Q_2}\left(n+3-\frac{4(k+1)}{n}\right)
		&=n^2-(k+1)^2
		+(3n-5k-5)\left(5-\frac{4(k+1)}{n}\right)\\
		&\quad+\left(2-\frac{4(k+1)}{n}\right)
		\left(5-\frac{4(k+1)}{n}\right)^2
		>0,
	\end{align*}
	where the last inequality follows from $n\ge2k+2$.
	Since $n+3-\tfrac{4(k+1)}{n}\ge n+1>n$ and $\varphi_{Q_2}(x)$ is strictly increasing on $[n,+\infty)$, we have $\rho_D(K_{k-1}\vee(K_{n-k}\cup K_1))=\lambda_1(Q_2)<n+3-\tfrac{4(k+1)}{n}$.
\end{proof}

Let $\kappa'(G)$ denote the edge connectivity of a connected graph $G$. We will use the following result of Li, Fan and Wang in the proof of Theorem \ref{thm:all1} $(iii)$.

\begin{lemma}[{\cite{Li2014}}]
	\label{lem:connectivity}
	Let $G$ be a connected graph of order $n\ge3$ with edge connectivity $\kappa'(G)=s$, where $1\le s\le n-2$. Then $\rho_D(G)\ge\rho_D\left(K_s\vee(K_{n-s-1}\cup K_1)\right)$, with strict inequality unless $G\cong K_s\vee(K_{n-s-1}\cup K_1)$.
\end{lemma}

We now prove Theorem \ref{thm:all1} $(iii)$.

\begin{proof}[Proof of Theorem \ref{thm:all1} (iii)]
	Suppose that $\tau(G)\le k-1$. By Theorem \ref{thm:packing}, there exists a partition $\mathcal P=\{V_1,\ldots,V_p\}$ of $V(G)$ into $2\le p\le n$ nonempty parts such that
	\begin{equation}
		\label{eq:small-partition}
		e_G(\mathcal P)\le k(p-1)-1.
	\end{equation}
	Let $n_i=|V_i|$ for $1\le i\le p$, then $\sum_{i=1}^{p}n_i=n$. We consider the following two cases.
	
	\noindent \textbf{Case 1.} $p\ge3$.
	
	Since adjacent vertices have distance one and nonadjacent vertices have distance at least two, we have $W(G)\ge\binom{n}{2}+e(\overline{G})= n(n-1)-e(G)$. By Lemma \ref{lem:sum} and Inequality \eqref{eq:small-partition}, we obtain
	\begin{align*}
		W(G)
		&\ge n(n-1)-\sum_{i=1}^{p}\binom{n_i}{2}-e_G(\mathcal P)\\
		&\ge n(n-1)-\binom{n-p+1}{2}-k(p-1)+1\\
		&=\frac{n(n+3)}2-2(k+1)
		+\frac{(p-3)(2n-2k-p-2)}2\\
		&\ge\frac{n(n+3)}2-2(k+1),
	\end{align*}
	where the last inequality follows from $p\ge3$ and $2n-2k-p-2\ge n-2k-2\ge0$. By Lemma \ref{lem:fan-bound}, we have
	\[
	\rho_D(G)\ge\frac{2W(G)}n
	\ge n+3-\frac{4(k+1)}n
	>\rho_D\left(K_{k-1}\vee(K_{n-k}\cup K_1)\right),
	\]
	which is contrary to the assumption.
	
	\noindent \textbf{Case 2.} $p=2$.
	
	Since $G$ is connected, Inequality \eqref{eq:small-partition} gives $1\le e_G(V_1,V_2)\le k-1$. Let $s=\kappa'(G)$. Then $1\le s\le k-1$. By Lemmas \ref{lem:connectivity} and \ref{lem:edge}, we have
	\[
		\rho_D(G)
		\ge\rho_D\left(K_s\vee(K_{n-s-1}\cup K_1)\right)\ge\rho_D\left(K_{k-1}\vee(K_{n-k}\cup K_1)\right).
	\]
	By the assumption of Theorem \ref{thm:all1} $(iii)$, we have $\rho_D(G)=\rho_D\left(K_{k-1}\vee(K_{n-k}\cup K_1)\right)$ and $s=k-1$, and Lemma \ref{lem:connectivity} gives $G\cong K_{k-1}\vee(K_{n-k}\cup K_1)$. This completes the proof.
\end{proof}

\section{Proof of Theorem \ref{thm:all2}}
\label{sec:degree}

The following upper bound will be used in the proof of Theorem \ref{thm:all2} $(i)$.

\begin{lemma}
	\label{lem:natural-bound}
	Let $k\ge2$ and $n\ge2k+3$ be integers. Then
	\begin{equation}
		\label{eq:natural-bound}
		\rho_D(F_{n,k})<n+5-\frac{6k+10}{n}.
	\end{equation}
\end{lemma}

\begin{proof}
	Let $u,v$ be the two vertices of $F_{n,k}$ outside the clique $K_{n-2}$. Let $A$ be the set of vertices in $K_{n-2}$ adjacent to $u$ or $v$, and let $B$ consist of the remaining vertices of $K_{n-2}$. Then $|A|=2k-2$ and $|B|=n-2k$. For the partition $\{\{u,v\},A,B\}$, the distance matrix $D(F_{n,k})$ has the equitable quotient matrix
	\[
	Q_3=
	\begin{pmatrix}
		1 & 3(k-1) & 2(n-2k)\\
		3 & 2k-3 & n-2k\\
		4 & 2k-2 & n-2k-1
	\end{pmatrix}.
	\]
	By Lemma \ref{lem:quotient}, $\rho_D(F_{n,k})=\lambda_1(Q_3)$. By a direct calculation, the characteristic polynomial of $Q_3$ is
	\[
	\varphi_{Q_3}(x)
	=x^3-(n-3)x^2-(8n-7k-8)x+nk-8n-2k^2+9k+6.
	\]
	For $x\ge n$, we have
	\[
	\varphi_{Q_3}'(x)
	=(x-n)(3x+n+6)+n^2-2n+7k+8>0.
	\]
	Thus $\varphi_{Q_3}(x)$ is strictly increasing on $[n,+\infty)$.
	
	Since $n\ge2k+3\ge7$, by a direct calculation, we obtain
	\begin{align*}
		\varphi_{Q_3}\left(n+5-\frac{6k+10}{n}\right)
		&=\left(4-\frac{6k+10}{n}\right)^3
		+\frac{5n+36}{6}\left(4-\frac{6k+10}{n}\right)^2\\
		&\quad+\frac{(3k+31)n+48}{9}
		\left(4-\frac{6k+10}{n}\right)
		+\frac{(6k+26)n-128}{9}\\
		&>0.
	\end{align*}
	Since $n+5-\tfrac{6k+10}{n}>n$, it follows that $\varphi_{Q_3}(x)>0$ for all $x\ge n+5-\tfrac{6k+10}{n}$. Therefore, $\rho_D(F_{n,k})=\lambda_1(Q_3)<n+5-\tfrac{6k+10}{n}$. This completes the proof.
\end{proof}

We now prove Theorem \ref{thm:all2} $(i)$.

\begin{proof}[Proof of Theorem \ref{thm:all2} (i)]
	Suppose that $\tau(G)\le k-1$. By Theorem \ref{thm:packing}, there exists a partition $\mathcal P=\{V_1,\ldots,V_p\}$ of $V(G)$ into $2\le p\le n$ nonempty parts such that
	\begin{equation}
		\label{eq:natural-partition}
		e_G(\mathcal P)\le k(p-1)-1.
	\end{equation}
	Let $n_i=|V_i|$ for $1\le i\le p$, then $\sum_{i=1}^{p}n_i=n$. Since adjacent vertices have distance one and nonadjacent vertices have distance at least two, we have $W(G)\ge\binom{n}{2}+e(\overline{G})= n(n-1)-e(G)$. We consider the following three cases.
	
	\noindent \textbf{Case 1.} $p\ge4$.
	
	By Lemma \ref{lem:sum} and Inequality \eqref{eq:natural-partition}, we obtain
	\begin{align*}
		W(G)
		&\ge n(n-1)-e(G) \\
		&\ge n(n-1)-\sum_{i=1}^{p}\binom{n_i}{2}-e_G(\mathcal P)\\
		&\ge n(n-1)-\binom{n-p+1}{2}-k(p-1)+1\\
		&=\frac{n(n+5)}2-3k-5
		+\frac{(p-4)(2n-2k-p-3)}2\\
		&\ge\frac{n(n+5)}2-3k-5,
	\end{align*}
	where the last inequality follows from $p\ge4$ and $2n-2k-p-3\ge n-2k-3\ge0$. By Lemma \ref{lem:natural-bound}, we have
	\[
	\rho_D(G)\ge\frac{2W(G)}n
	\ge n+5-\frac{6k+10}{n}
	>\rho_D(F_{n,k}),
	\]
	which is contrary to the assumption.
	
	\noindent \textbf{Case 2.} $p=2$.
	
	By Inequality \eqref{eq:natural-partition}, we have $e_G(V_1,V_2)\le k-1\le \delta -1$, Lemma \ref{lem:cut} gives $n_1,n_2\ge \delta +1\ge k+1$. By Lemma \ref{lem:sum} and Inequality \eqref{eq:natural-partition}, we obtain
	\begin{align*}
		W(G)
		&\ge n(n-1)-\sum_{i=1}^{2}\binom{n_i}{2}-e_G(\mathcal P)\\
		&\ge n(n-1)-\binom{n-k-1}{2}-\binom{k+1}{2}-k+1\\
		&=\frac{n(n+5)}2-3k-5
		+(k-2)(n-2k-3)+k^2-k-1\\
		&>\frac{n(n+5)}2-3k-5,
	\end{align*}
	where the last inequality follows from $n\ge2k+3$ and $k\ge2$. By Lemma \ref{lem:natural-bound}, we have
	\[
	\rho_D(G)\ge\frac{2W(G)}n
	>n+5-\frac{6k+10}{n}
	>\rho_D(F_{n,k}),
	\]
	which is contrary to the assumption.
	
	\noindent \textbf{Case 3.} $p=3$.
	
	Without loss of generality, assume that $n_1\le n_2\le n_3$. Suppose that $n_2\ge2$. Then $n_1\ge1$ and $n_2,n_3\ge2$. By Lemma \ref{lem:sum} and Inequality \eqref{eq:natural-partition}, we obtain
	\begin{align*}
		W(G)
		&\ge n(n-1)-\sum_{i=1}^{3}\binom{n_i}{2}-e_G(\mathcal P)\\
		&\ge n(n-1)-\binom{n-3}{2}-1-(2k-1)\\
		&=\frac{n(n+5)}2-2k-6\\
		&>\frac{n(n+5)}2-3k-5.
	\end{align*}
	By Lemma \ref{lem:natural-bound}, we have
	\[
	\rho_D(G)\ge\frac{2W(G)}n
	>n+5-\frac{6k+10}{n}
	>\rho_D(F_{n,k}),
	\]
	which is contrary to the assumption. Thus $n_1=n_2=1$.
	
	Let $\mathcal P=\{\{u\},\{v\},R\}$, where $|R|=n-2$. If $uv\notin E(G)$, then $e_G(\mathcal P)=d_G(u)+d_G(v)\ge2k$, which is contrary to Inequality \eqref{eq:natural-partition}. Hence $uv\in E(G)$, and
	\[
	2k\le d_G(u)+d_G(v)=e_G(\mathcal P)+1\le2k.
	\]
	Hence $d_G(u)=d_G(v)=k$ and $e_G(\mathcal P)=2k-1$.
	
	Let $H$ be obtained from $G$ by adding all missing edges within $R$. By Lemma \ref{lem:edge}, we have $\rho_D(G)\ge\rho_D(H)$, with strict inequality unless $G\cong H$. Let $S=N_H(u)\cap R$ and $T=N_H(v)\cap R$. Then $|S|=|T|=k-1$.
	
	Suppose that $S\cap T\ne\emptyset$. Choose $w\in S\cap T$. Since $|R\setminus(S\cup T)|
	\ge n-2-2(k-1)=n-2k>0$, there exists $z\in R\setminus(S\cup T)$. Let $H'=H-uw+uz$, then $d_{H'}(u)=d_{H'}(v)=k$ and $|N_{H'}(u)\cap N_{H'}(v)|=|S\cap T|-1$. The only distances that change are
	$d_H(u,w)=1$, $d_{H'}(u,w)=2$ and
	$d_H(u,z)=2$, $d_{H'}(u,z)=1$.
	
	Let $\mathbf x=(x_t)_{t\in V(H')}$ be a positive unit eigenvector of $D(H')$ corresponding to $\rho_D(H')$. 
	
	We now prove that $x_w=x_z$. Let $\rho=\rho_D(H')$ and $S'=(S\setminus\{w\})\cup\{z\}$. Then $N_{H'}(u)\cap R=S'$,$N_{H'}(v)\cap R=T$. For $t\in S'\setminus(T\cup\{z\})$, the eigenvector equations at $t$ and $z$ in $H'$ are
	\[
	\rho x_t=x_u+2x_v+x_z+\sum_{r\in R\setminus\{t,z\}}x_r,
	\qquad
	\rho x_z=x_u+2x_v+x_t+\sum_{r\in R\setminus\{t,z\}}x_r.
	\]
	Subtracting these two equations, we obtain $(\rho+1)(x_t-x_z)=0$. Hence $x_t=x_z$ for every $t\in S'\setminus T$. Similarly, $x_t=x_w$ for every $t\in T\setminus S'$.
	
	Since $|S'|=|T|=k-1$, we have $|S'\setminus T|=|T\setminus S'|=k-1-|S'\cap T|$. Comparing the eigenvector equations at $w,z$ and at $u,v$, respectively, we obtain
	\begin{align*}
		(\rho+1)(x_w-x_z)&=x_u-x_v,\\
		(\rho+1)(x_u-x_v)
		&=\sum_{t\in T\setminus S'}x_t-\sum_{t\in S'\setminus T}x_t=(k-1-|S'\cap T|)(x_w-x_z).
	\end{align*}
	Hence,
	\[
	\bigl((\rho+1)^2-k+1+|S'\cap T|\bigr)(x_w-x_z)=0.
	\]
	Since $\rho\ge \frac{2W(H')}{n}\ge\frac2n \binom n2 =n-1$. Therefore,
	\[
	\bigl((\rho+1)^2-k+1+|S'\cap T|\bigr)
	\ge n^2-k+1>0,
	\]
	and hence $x_w=x_z$.
	
	By the Rayleigh quotient principle, we have
	\[
	\rho_D(H)-\rho_D(H')\ge\mathbf x^{\mathsf T}\bigl(D(H)-D(H')\bigr)\mathbf x=2x_u(x_z-x_w)=0.
	\]
	If $\rho_D(H)=\rho_D(H')$, then $\mathbf x$ is also an eigenvector of $D(H)$ corresponding to $\rho_D(H)$. Comparing the eigenvector equations at $z$, and noting that $d_H(z,u)=2$ and $d_{H'}(z,u)=1$, while $d_H(z,a)=d_{H'}(z,a)$ for every $a\ne u,z$, we obtain
	\begin{align*}
		(\rho_D(H)-\rho_D(H'))x_z
		&=\sum_{a\in V(H)}\bigl(d_H(z,a)-d_{H'}(z,a)\bigr)x_a\\
		&=\bigl(d_H(z,u)-d_{H'}(z,u)\bigr)x_u\\
		&=x_u>0,
	\end{align*}
	a contradiction. Thus $\rho_D(H)>\rho_D(H')$.
	
	Repeating the above switching operation until the neighborhoods of $u$ and $v$ in $R$ are disjoint, we obtain a graph isomorphic to $F_{n,k}$. Since the distance spectral radius strictly decreases at each step, it follows that
	\[
	\rho_D(G)\ge \rho_D(H)>\rho_D(F_{n,k}),
	\]
	which contradicts the assumption. Hence $S\cap T=\emptyset$ and $H\cong F_{n,k}$. By Lemma \ref{lem:edge} and the assumption, we have $G\cong F_{n,k}$. This completes the proof.
\end{proof}

We next prove Theorem \ref{thm:all2} $(ii)$.

\begin{proof}[Proof of Theorem \ref{thm:all2} (ii)]
	Suppose that $\tau(G)\le k-1$. By Theorem \ref{thm:packing}, there exists a partition $\mathcal P=\{V_1,\ldots,V_p\}$ of $V(G)$ into $2\le p\le n$ nonempty parts such that
	\begin{equation}
		\label{eq:critical-partition}
		e_G(\mathcal P)\le k(p-1)-1.
	\end{equation}
	Let $n_i=|V_i|$ for $1\le i\le p$, so that $\sum_{i=1}^{p}n_i=n$. 
	
	Suppose that $p=2$. By Inequality \eqref{eq:critical-partition}, we have $e_G(V_1,V_2)\le k-1\le\delta-1$. By Lemma \ref{lem:cut}, $n_1,n_2\ge\delta+1\ge k+1$. Since $n=2k+2$, we obtain $n_1=n_2=k+1$. Hence,
	\begin{align*}
		W(G)
		&\ge n(n-1)-e(G)\\
		&\ge n(n-1)-2\binom{k+1}{2}-k+1\\
		&=\frac{n(n+1)}2+k(k-1).
	\end{align*}
	It follows that
	\[
	\rho_D(G)\ge\frac{2W(G)}n
	\ge n+1+\frac{2k(k-1)}n>n+1,
	\]
	which is contrary to the assumption. Thus $p\ge3$.
	
	By Lemma \ref{lem:sum} and Inequality \eqref{eq:critical-partition}, we obtain
	\begin{align*}
		e(G)
		&\le\sum_{i=1}^{p}\binom{n_i}{2}+e_G(\mathcal P)\\
		&\le\binom{n-p+1}{2}+k(p-1)-1\\
		&=\frac{n(n-3)}2-\frac{(p-3)(n-p)}2\\
		&\le\frac{n(n-3)}2,
	\end{align*}
	where we use $n=2k+2$ and $3\le p\le n$. 
	Together with the assumption, we obtain
	\begin{equation}
		\label{eq:critical-equality}
		n+1\ge\rho_D(G)\ge\frac{2W(G)}n
		\ge2(n-1)-\frac{2e(G)}n
		\ge n+1.
	\end{equation}
	Thus all inequalities in \eqref{eq:critical-equality} are equalities, and $W(G)=n(n-1)-e(G)$ implies that $d_G(u,v)=2$ for every $uv\in E(\overline G)$. Since $\frac{\mathbf1^{\mathsf T}D(G)\mathbf1} {\mathbf1^{\mathsf T}\mathbf1} =\frac{2W(G)}n =\rho_D(G)$, we have $D(G)\mathbf1=(n+1)\mathbf1$. Hence, for every $v\in V(G)$,
	\[
		n+1=\sum_{w\in V(G)}d_G(v,w)=d_G(v)+2\bigl(n-1-d_G(v)\bigr)
		=n-1+d_{\overline G}(v).
	\]
	Hence $d_{\overline G}(v)=2$ for every $v\in V(G)$. Thus $\overline G$ is a disjoint union of cycles, and $G\cong H$ for some $H\in\mathcal F_n$. This completes the proof.
\end{proof}

\section{Proof of Theorem \ref{thm:strong}}

In this section, we prove Theorem \ref{thm:strong}. We first give a lemma.

\begin{lemma}
	\label{lem:matching}
	Let $k\ge2$, $\delta\ge k$ and $n\ge2\delta+2$ be integers. If $H$ is obtained from $K_{\delta+1}\cup K_{n-\delta-1}$ by adding $k-1$ edges between the two complete graphs, then $\rho_D(H)\ge\rho_D(M_{n,\delta+1}^{k-1})$, with equality if and only if the added edges form a matching.
\end{lemma}

\begin{proof}
	Let $A$ and $B$ be the vertex sets of $K_{\delta+1}$ and $K_{n-\delta-1}$, respectively. Suppose that the added edges have $a$ distinct endpoints in $A$ and $b$ distinct endpoints in $B$. Then $1\le a,b\le k-1$. Choose $A_1\subseteq A$ and $B_1\subseteq B$ with $|A_1|=|B_1|=k-1$. These sets contain all endpoints of the added edges in $A$ and $B$, respectively. We obtain $M_{n,\delta+1}^{k-1}$ by replacing the added edges with a matching of size $k-1$ between $A_1$ and $B_1$.
	
	Consider the partition $\{A_1,A\setminus A_1,B_1,B\setminus B_1\}$, the distance matrix $D(M_{n,\delta+1}^{k-1})$ has the equitable quotient matrix
	\begin{equation}
		\label{eq:matching-quotient}
		Q_{n,\delta+1}^{k-1}=
		\begin{pmatrix}
			k-2 & \delta-k+2 & 2k-3 & 2(n-\delta-k)\\
			k-1 & \delta-k+1 & 2k-2 & 3(n-\delta-k)\\
			2k-3 & 2(\delta-k+2) & k-2 & n-\delta-k\\
			2k-2 & 3(\delta-k+2) & k-1 & n-\delta-k-1
		\end{pmatrix}.
	\end{equation}
	Let $\mathbf x=(x_1,x_2,x_3,x_4)^{\mathsf T}$ be a positive eigenvector of $Q_{n,\delta+1}^{k-1}$ corresponding to its largest eigenvalue $\lambda_1(Q_{n,\delta+1}^{k-1})$, normalized so that $(k-1)x_1^2+(\delta-k+2)x_2^2
	+(k-1)x_3^2+(n-\delta-k)x_4^2=1$. Define $\mathbf z=(z_v)_{v\in V(H)}$ by $z_v=x_1$ for $v\in A_1$, $z_v=x_2$ for $v\in A\setminus A_1$, $z_v=x_3$ for $v\in B_1$, and $z_v=x_4$ for $v\in B\setminus B_1$. Then $\mathbf z^{\mathsf T}\mathbf z
	=(k-1)x_1^2+(\delta-k+2)x_2^2
	+(k-1)x_3^2+(n-\delta-k)x_4^2=1$. By Lemma \ref{lem:quotient}, we have
	\[
	D(M_{n,\delta+1}^{k-1})\mathbf z
	=\lambda_1(Q_{n,\delta+1}^{k-1})\mathbf z
	=\rho_D(M_{n,\delta+1}^{k-1})\mathbf z.
	\]
	
	For nonadjacent vertices $u\in A$ and $v\in B$, we have $d_H(u,v)=2$ if at least one of them is incident with an added edge, and $d_H(u,v)=3$ otherwise. Both $H$ and $M_{n,\delta+1}^{k-1}$ have $k-1$ edges between $A_1$ and $B_1$. In $H$, exactly $(k-1-a)(k-1-b)$ pairs in $A_1\times B_1$ have distance three, while all remaining nonadjacent pairs in these two sets have distance two. In $M_{n,\delta+1}^{k-1}$, every nonadjacent pair in $A_1\times B_1$ has distance two.
	
	Distances within $A$ and $B$, and those between $A\setminus A_1$ and $B\setminus B_1$, are the same in $H$ and $M_{n,\delta+1}^{k-1}$. By the Rayleigh quotient principle, we obtain
	\begin{align*}
		\rho_D(H)-\rho_D(M_{n,\delta+1}^{k-1})
		&\ge\mathbf z^{\mathsf T}\bigl(D(H)-D(M_{n,\delta+1}^{k-1})\bigr)\mathbf z\\
		&=2x_1x_3\sum_{u\in A_1}\sum_{v\in B_1}
		\bigl(d_H(u,v)-d_{M_{n,\delta+1}^{k-1}}(u,v)\bigr)\\
		&\quad+2x_1x_4\sum_{u\in A_1}\sum_{v\in B\setminus B_1}
		\bigl(d_H(u,v)-d_{M_{n,\delta+1}^{k-1}}(u,v)\bigr)\\
		&\quad+2x_2x_3\sum_{u\in A\setminus A_1}\sum_{v\in B_1}
		\bigl(d_H(u,v)-d_{M_{n,\delta+1}^{k-1}}(u,v)\bigr)\\
		&=2x_1x_3\Bigl[2(k-1)^2-(k-1)+(k-1-a)(k-1-b)
		-\bigl(2(k-1)^2-(k-1)\bigr)\Bigr]\\
		&\quad+2x_1x_4\Bigl[\bigl(2a+3(k-1-a)\bigr)(n-\delta-k)
		-2(k-1)(n-\delta-k)\Bigr]\\
		&\quad+2x_2x_3\Bigl[\bigl(2b+3(k-1-b)\bigr)(\delta-k+2)
		-2(k-1)(\delta-k+2)\Bigr]\\
		&=2x_1x_3(k-1-a)(k-1-b)
		+2x_1x_4(k-1-a)(n-\delta-k)\\
		&\quad+2x_2x_3(\delta-k+2)(k-1-b)\\
		&\ge0.
	\end{align*}
	If the added edges do not form a matching, then $a<k-1$ or $b<k-1$. Since $\delta-k+2,n-\delta-k>0$ and $x_1,x_2,x_3,x_4>0$, we have $\rho_D(H)>\rho_D(M_{n,\delta+1}^{k-1})$. If the added edges form a matching, then $H\cong M_{n,\delta+1}^{k-1}$. This completes the proof.
\end{proof}

We now prove Theorem \ref{thm:strong}.

\begin{proof}[Proof of Theorem \ref{thm:strong}]
	Suppose that $\tau(G)\le k-1$. By Theorem \ref{thm:packing}, there exists a partition $\mathcal P=\{V_1,\ldots,V_p\}$ of $V(G)$ into $2\le p\le n$ nonempty parts such that
	\begin{equation}
		\label{eq:strong-partition}
		e_G(\mathcal P)\le k(p-1)-1.
	\end{equation}
	Choose $\mathcal P$ so that $p$ is as small as possible. Let $n_i=|V_i|$ for $1\le i\le p$, so that $\sum_{i=1}^{p}n_i=n$. We consider the following two cases.
	
	\noindent \textbf{Case 1.} $p\ge3$.
	
	Let $r_i=e_G(V_i,V(G)\setminus V_i)$ for $1\le i\le p$. Suppose that $r_i\le k-1=k(2-1)-1$ for some $i$. Then the partition $\{V_i,V(G)\setminus V_i\}$ also satisfies Inequality \eqref{eq:strong-partition} and has only two parts, contrary to the choice of $\mathcal P$. Hence $r_i\ge k$ for every $1\le i\le p$.
	
	Without loss of generality, assume that $r_1\le r_2\le\cdots\le r_p$. If $r_3\ge2k$, then
	\[
	2e_G(\mathcal P)=\sum_{i=1}^{p}r_i
	\ge2k(p-2)+2k=2k(p-1),
	\]
	contrary to Inequality \eqref{eq:strong-partition}. Hence,
	\[
	k\le r_1\le r_2\le r_3\le2k-1\le\delta-1.
	\]
	By Lemma \ref{lem:cut}, we have $n_1,n_2,n_3\ge\delta+1$.
	
	Let $V'=V(G)\setminus(V_1\cup V_2)$ and $n'=|V'|$. Then $\{V_1,V_2,V'\}$ is a partition of $V(G)$, $n_1,n_2, n'\ge \delta+1\ge 3(2k-1)$ and $n_1+n_2+n'=n\ge 3(\delta+1)$. 
	
	Let $V_a,V_b$ be any two distinct parts of $\{V_1,V_2,V'\}$. Without loss of generality, assume that $V_a=V_1$ or $V_a=V_2$. Then $e_G(V_a,V(G)\setminus V_a)\le2k-1$. Hence at least $|V_a|-(2k-1)$ vertices of $V_a$ have no neighbor outside $V_a$. Also, since $e_G(V_a,V_b)\le e_G(V_a,V(G)\setminus V_a)\le2k-1$, at least $|V_b|-(2k-1)$ vertices of $V_b$ have no neighbor in $V_a$. Let $u$ and $v$ be any two such vertices in $V_a$ and $V_b$, respectively. Then $uv\notin E(G)$ and $N_G(u)\cap N_G(v)=\emptyset$, so $d_G(u,v)\ge3$. Since all other nonadjacent pairs have distance at least two, we obtain
	\begin{align*}
		\sum_{u\in V_a}\sum_{v\in V_b}d_G(u,v)
		&\ge e_G(V_a,V_b)
		+2\bigl(|V_a||V_b|-e_G(V_a,V_b)\bigr)+(|V_a|-2k+1)(|V_b|-2k+1)\\
		&=2|V_a||V_b|-e_G(V_a,V_b)
		+(|V_a|-2k+1)(|V_b|-2k+1)\\
		&\ge2|V_a||V_b|-(2k-1)
		+(|V_a|-2k+1)(|V_b|-2k+1).
	\end{align*}
	Since $|V_a|,|V_b|\ge3(2k-1)$ and $k\ge2$, we have
	\begin{equation}
		\label{eq:strong-average}
		\begin{aligned}
			\frac{1}{|V_a||V_b|}
			\sum_{u\in V_a}\sum_{v\in V_b}d_G(u,v)
			&\ge2-\frac{2k-1}{|V_a||V_b|}
			+\left(1-\frac{2k-1}{|V_a|}\right)
			\left(1-\frac{2k-1}{|V_b|}\right)\\
			&\ge2-\frac{1}{9(2k-1)}
			+\left(1-\frac13\right)^2\\
			&\ge2-\frac1{27}+\frac49
			=\frac{65}{27}>\frac{12}{5}.
		\end{aligned}
	\end{equation}
	
	Consider the matrix
	\[
	Q=
	\begin{pmatrix}
		n_1-1 & \frac{12}{5}n_2 & \frac{12}{5}n'\\
		\frac{12}{5}n_1 & n_2-1 & \frac{12}{5}n'\\
		\frac{12}{5}n_1 & \frac{12}{5}n_2 & n'-1
	\end{pmatrix}.
	\]
	Let $\mathbf x=(x_1,x_2,x_3)^{\mathsf T}$ be a positive eigenvector of $Q$ corresponding to its largest eigenvalue $\lambda_1(Q)$, normalized so that $n_1x_1^2+n_2x_2^2+n'x_3^2=1$. Define $\mathbf z=(z_v)_{v\in V(G)}$ by $z_v=x_1$ for $v\in V_1$, $z_v=x_2$ for $v\in V_2$ and $z_v=x_3$ for $v\in V'$. Then $\mathbf z^{\mathsf T}\mathbf z=n_1x_1^2+n_2x_2^2+n'x_3^2=1$.
	
	Since distinct vertices in the same part have distance at least one, by Inequality \eqref{eq:strong-average} and the Rayleigh quotient principle, we obtain
	\begin{align*}
		\rho_D(G)
		&\ge\mathbf z^{\mathsf T}D(G)\mathbf z\\
		&=x_1^2\sum_{u,v\in V_1}d_G(u,v)
		+x_2^2\sum_{u,v\in V_2}d_G(u,v)
		+x_3^2\sum_{u,v\in V'}d_G(u,v)\\
		&\quad+2x_1x_2\sum_{u\in V_1}\sum_{v\in V_2}d_G(u,v)
		+2x_1x_3\sum_{u\in V_1}\sum_{v\in V'}d_G(u,v)
		+2x_2x_3\sum_{u\in V_2}\sum_{v\in V'}d_G(u,v)\\
		&\ge n_1(n_1-1)x_1^2+n_2(n_2-1)x_2^2+n'(n'-1)x_3^2+2\cdot\frac{12}{5}
		\bigl(n_1n_2x_1x_2+n_1n'x_1x_3+n_2n'x_2x_3\bigr)\\
		&=n_1x_1\left((n_1-1)x_1+\frac{12}{5}n_2x_2+\frac{12}{5}n'x_3\right)
		+n_2x_2\left(\frac{12}{5}n_1x_1+(n_2-1)x_2+\frac{12}{5}n'x_3\right)\\
		&\quad+n'x_3\left(\frac{12}{5}n_1x_1+\frac{12}{5}n_2x_2+(n'-1)x_3\right)\\
		&=n_1x_1(Q\mathbf x)_1+n_2x_2(Q\mathbf x)_2+n'x_3(Q\mathbf x)_3\\
		&=\lambda_1(Q)\bigl(n_1x_1^2+n_2x_2^2+n'x_3^2\bigr)\\
		&=\lambda_1(Q).
	\end{align*}
	
	We next compare $\lambda_1(Q)$ with $\rho_D(M_{n,\delta+1}^{k-1})$. 
	
	Let $V(K_{\delta+1})=\{v_1,\ldots,v_{\delta+1}\}$ and $V(K_{n-\delta-1})=\{v_{\delta+2},\ldots,v_n\}$, and write $D(M_{n,\delta+1}^{k-1})=(d_{ij})$. Define $E=(e_{ij})$ by $e_{ij}=d_{ij}$ if $1\le i,j\le\delta+1$ or $\delta+2\le i,j\le n$, and $e_{ij}=3$ otherwise. Then $E\ge D(M_{n,\delta+1}^{k-1})$ and $E\ne D(M_{n,\delta+1}^{k-1})$. With respect to the partition $\{V(K_{\delta+1}),V(K_{n-\delta-1})\}$, the matrix $E$ has the equitable quotient matrix
	\[
	Q_E=
	\begin{pmatrix}
		\delta & 3(n-\delta-1)\\
		3(\delta+1) & n-\delta-2
	\end{pmatrix}.
	\]
	By a direct calculation, the characteristic polynomial of $Q_E$ is
	\[
	\varphi_{Q_E}(x)
	=(x+1)^2-n(x+1)-8(\delta+1)(n-\delta-1).
	\]
	By a direct calculation, let
	\[
	\theta=\lambda_1(Q_E)+1
	=\frac{n+\sqrt{n^2+32(\delta+1)(n-\delta-1)}}2.
	\]
	By the Perron--Frobenius theorem and Lemma \ref{lem:quotient}, we have
	\begin{equation}
		\label{eq:strong-upper}
		\rho_D(M_{n,\delta+1}^{k-1})
		<\lambda_1(E)=\lambda_1(Q_E)=\theta-1.
	\end{equation}
	Since $\varphi_{Q_E}(\theta-1)=0$, we have $\theta^2-n\theta=8(\delta+1)(n-\delta-1)\le 2n^2$, hence $\theta\le2n$.
	
	Let $\varphi_Q(x)$ be the characteristic polynomial of $Q$. By a direct calculation, we have
	\[
	\varphi_Q(x-1)=x^3-nx^2-\frac{119}{25}(n_1n_2+n_1n'+n_2n')x-\frac{1421}{125}n_1n_2n'.
	\]
	
	By Lemma \ref{lem:sum} and $n_1,n_2,n'\ge\delta+1$, we obtain
	\begin{align*}
		n_1n_2+n_1n'+n_2n'
		&=\binom n2-\binom{n_1}{2}-\binom{n_2}{2}-\binom{n'}{2}\\
		&\ge\binom n2-2\binom{\delta+1}{2}-\binom{n-2\delta-2}{2}\\
		&=(\delta+1)(2n-3\delta-3).
	\end{align*}
	Let $n_1=\delta+1+a$, $n_2=\delta+1+b$, $n'=\delta+1+c$, where \(a,b,c\ge 0\). Since \(n=n_1+n_2+n'\), we have
	\begin{align*}
		&n_1n_2n'-(n-2(\delta+1))(\delta+1)^2\\
		&=(\delta+1+a)(\delta+1+b)(\delta+1+c)
		-(\delta+1+a+b+c)(\delta+1)^2\\
		&=(\delta+1)(ab+bc+ca)+abc\ge 0.
	\end{align*}
	
	Using the above inequalities and $\theta^2-n\theta=8(\delta+1)(n-\delta-1)$, we obtain
	\begin{align*}
		\varphi_Q(\theta-1)
		&\le\theta^3-n\theta^2
		-\frac{119}{25}(\delta+1)(2n-3\delta-3)\theta
		-\frac{1421}{125}(\delta+1)^2(n-2\delta-2)\\
		&=8(\delta+1)(n-\delta-1)\theta
		-\frac{119}{25}(\delta+1)(2n-3\delta-3)\theta
		-\frac{1421}{125}(\delta+1)^2(n-2\delta-2)\\
		&=\frac{(\delta+1)\bigl(157(\delta+1)-38n\bigr)\theta}{25}
		-\frac{1421(\delta+1)^2(n-2\delta-2)}{125}\\
		&\le-\frac{131(\delta+1)^3}{125}<0,
	\end{align*}
	which the last inequalities follows from $\theta\le2n$ and $n\ge3\delta+3$. Hence $\varphi_Q(x)$ has a real zero greater than $\theta-1=\lambda_1(Q_E)$. Thus $\lambda_1(Q)>\lambda_1(Q_E)$. Together with Inequality \eqref{eq:strong-upper}, we have
	\[
	\rho_D(G)\ge\lambda_1(Q)>\lambda_1(Q_E)>\rho_D(M_{n,\delta+1}^{k-1}),
	\]
	which is contrary to the assumption.
	
	\noindent \textbf{Case 2.} $p=2$.
	
	By Inequality \eqref{eq:strong-partition}, we have $e_G(V_1,V_2)\le k-1\le\delta-1$. By Lemma \ref{lem:cut}, $n_1,n_2\ge\delta+1$. Without loss of generality, assume that $n_1\le n_2$.
	
	Let $H$ be obtained from $G$ by adding all missing edges within $V_1$ and $V_2$, and then adding edges between $V_1$ and $V_2$ until there are exactly $k-1$ such edges. By Lemmas \ref{lem:edge} and \ref{lem:matching}, together with the assumption, we obtain
	\begin{equation}
		\label{eq:strong-two-parts}
		\rho_D(M_{n,\delta+1}^{k-1})
		\ge\rho_D(G)\ge\rho_D(H)
		\ge\rho_D(M_{n,n_1}^{k-1}).
	\end{equation}
	
	We claim that $n_1=\delta+1$. Suppose that $n_1>\delta+1$. Let $\varphi_{\delta+1}(x)$ and $\varphi_{n_1}(x)$ be the characteristic polynomials of $Q_{n,\delta+1}^{k-1}$ and $Q_{n,n_1}^{k-1}$ in \eqref{eq:matching-quotient}, respectively. By a direct calculation, the characteristic polynomials of $Q_{n,n_1}^{k-1}$ and $Q_{n,\delta+1}^{k-1}$ are
	\begin{align*}
		\varphi_{n_1}(x)
		&=x^4-(n-4)x^3
		+\bigl((5k-8)n-5k^2+14k-4-8n_1n_2\bigr)x^2\\
		&\quad+\bigl((2k-18)n_1n_2+(-2k^2+14k-14)n
		+2k^3-16k^2+32k-16\bigr)x\\
		&\quad+(3k^2-14k+11)n_1n_2
		+(-3k^3+17k^2-28k+14)n\\
		&\quad+3k^4-20k^3+45k^2-40k+12
	\end{align*}
	and
	\begin{align*}
		\varphi_{\delta+1}(x)
		&=x^4-(n-4)x^3+\bigl((5k-8)n-5k^2+14k-4
		-8(\delta+1)(n-\delta-1)\bigr)x^2\\
		&\quad+\bigl((2k-18)(\delta+1)(n-\delta-1)
		+(-2k^2+14k-14)n+2k^3-16k^2+32k-16\bigr)x\\
		&\quad+(3k^2-14k+11)(\delta+1)(n-\delta-1)
		+(-3k^3+17k^2-28k+14)n\\
		&\quad+3k^4-20k^3+45k^2-40k+12,
	\end{align*}
	respectively. Since $n_1+n_2=n$ and $n_2\ge n_1>\delta+1$, for $x\ge k-1$, we obtain
	\begin{align*}
		\varphi_{n_1}(x)-\varphi_{\delta+1}(x)
		&=\bigl(n_1n_2-(\delta+1)(n-\delta-1)\bigr)
		\bigl(-8x^2+(2k-18)x+3k^2-14k+11\bigr)\\
		&=\bigl(n_1n_2-(\delta+1)(n_1+n_2)+(\delta+1)^2\bigr)\bigl(-8x^2+(2k-18)x+3k^2-14k+11\bigr)\\
		&=(n_1-\delta-1)(n_2-\delta-1)\bigl(-8x^2+2(k-1)x+3(k-1)^2-16x-8(k-1)\bigr)\\
		&\le(n_1-\delta-1)(n_2-\delta-1)
		\bigl(-8x^2+2x^2+3x^2-16x-8(k-1)\bigr)\\
		&=(n_1-\delta-1)(n_2-\delta-1)
		\bigl(-3x^2-16x-8(k-1)\bigr)\\
		&<0.
	\end{align*}
	
	Let $\rho=\rho_D(M_{n,\delta+1}^{k-1})$, we have
	\[
	\rho\ge
	\frac{\mathbf1^{\mathsf T}D(M_{n,\delta+1}^{k-1})\mathbf1}
	{\mathbf1^{\mathsf T}\mathbf1}
	\ge\frac{n(n-1)}n=n-1>k-1.
	\]
	By Lemma \ref{lem:quotient}, $\varphi_{\delta+1}(\rho)=0$, and hence $\varphi_{n_1}(\rho)<\varphi_{\delta+1}(\rho)=0$. Hence $\varphi_{n_1}(x)$ has a real zero greater than $\rho$. Thus,
	\[
	\rho_D(M_{n,n_1}^{k-1})
	=\lambda_1(Q_{n,n_1}^{k-1})
	>\rho_D(M_{n,\delta+1}^{k-1}),
	\]
	which is contrary to Inequality \eqref{eq:strong-two-parts}. Thus $n_1=\delta+1$, and all inequalities in \eqref{eq:strong-two-parts} are equalities. By Lemma \ref{lem:edge}, $G\cong H$. By Lemma \ref{lem:matching}, the edges between $V_1$ and $V_2$ form a matching. Hence $G\cong M_{n,\delta+1}^{k-1}$. This completes the proof.
\end{proof}

\section{Concluding remarks}

In this paper, we obtain sharp distance spectral radius conditions for $\tau(G)\ge k$, covering every possible order $n\ge2k$. We also obtain sharp conditions under different minimum degree assumptions.

We recall the following spectral radius condition established by
Fan et al. \cite{Fan2023}.

\begin{theorem}[Fan et al. \cite{Fan2023}]
	\label{thm:fan2023}
	Let $k\ge2$ be an integer and let $G$ be a connected graph of order
	$n\ge2\delta+3$ with minimum degree $\delta\ge2k$.
	If $\rho(G)\ge\rho(B_{n,\delta+1}^{k-1})$,
	then $\tau(G)\ge k$, unless $G\cong B_{n,\delta+1}^{k-1}$,
	where $B_{n,\delta+1}^{k-1}$ is obtained from
	$K_{\delta+1}\cup K_{n-\delta-1}$ by adding $k-1$ edges
	joining one vertex of $K_{\delta+1}$ to $k-1$ distinct vertices
	of $K_{n-\delta-1}$.
\end{theorem}

It's interesting to compare the extremal graphs $M_{n,\delta+1}^{k-1}$ in
Theorem \ref{thm:strong} and $B_{n,\delta+1}^{k-1}$ in
Theorem \ref{thm:fan2023}.
Both are obtained from $K_{\delta+1}\cup K_{n-\delta-1}$
by adding $k-1$ edges between the two cliques.
In $M_{n,\delta+1}^{k-1}$ and $B_{n,\delta+1}^{k-1}$, the $k-1$ edges form a matching and a star, representing the most dispersed and the most concentrated arrangements, respectively.

In addition, the minimum degree bound $\delta\ge6k-4$ in
Theorem \ref{thm:strong} may not be best possible.
Determining a better minimum degree bound is left for future work.

\section*{Declaration of competing interest}
\quad\quad The authors declare that they have no known competing financial interests or personal relationships that could have appeared to influence the work reported in this paper.

\section*{Data availability}
\quad\quad No data was used for the research described in the article.

\end{document}